\documentclass[reqno]{amsart}
\usepackage{amssymb,latexsym}
\usepackage{amsmath}
\usepackage{amsthm}
\usepackage{graphicx}
\usepackage{hyperref}
\usepackage{titletoc}
\numberwithin{equation}{section}
\newtheorem{theorem}{Theorem}[section]
\newtheorem{prop}[theorem]{Proposition}
\newtheorem{lemma}[theorem]{Lemma}
\newtheorem{corollary}[theorem]{Corollary}

\theoremstyle{definition}

\newcommand{\Z}{\mathbb Z}

\newcommand{\R}{\mathbb R}
\newcommand{\N}{\mathbb N}

\newcommand{\intBd}{\partial_i}
\newcommand{\outBd}{\partial_o}
\newcommand{\diam}{\textrm{diam}}
\newcommand{\Rtheta}{\theta^-}
\newtheorem{remark}{Remark}[section]

\newcommand{\ssubset}{\subset\subset}

\newcommand{\BCap}{\text{BCap}}
\newcommand{\EsC}{\text{Es}}
\newcommand{\Esc}{\text{Esc}}

\newcommand{\KRW}{\mathbf k}
\newcommand{\SRW}{\mathbf s}
\newcommand{\BRW}{\mathbf b}
\newcommand{\Snake}{\mathcal S}

\newcommand{\rp}{\widetilde{r}}

\newcommand{\pS}{\mathbf p}
\newcommand{\rS}{\mathbf r}

\newcommand{\DeFine}{\doteq}
\newcommand{\dist}{\rho}

\newcommand{\Ball}{\mathcal C}
\newcommand{\BallE}{\mathcal B}

\newcommand{\Rad}{\text{Rad}}
\newcommand{\Hm}{\mathcal{H}}

\newcommand{\Num}{\mathcal{N}}

\newcommand{\pr}{\Xi}
\newcommand{\subt}{\rho}

\newcommand{\EE}{\mathcal{E}}

\begin{document}
\title[On CBRW III]{On the critical branching random walk III: the critical dimension}
\author{Qingsan Zhu}
\address{~Department of Mathematics, University of British
Columbia,
Vancouver, BC V6T 1Z2, Canada}
\email{qszhu@math.ubc.ca}
\date{}
\maketitle

\begin{abstract}
In this paper, we study the critical branching random walk in the critical dimension, $\Z^4$. We provide the asymptotics of the probability of visiting a fixed finite subset and the range of the critical branching random walk conditioned on the total number of offsprings. We also prove that conditioned on visiting, the first visiting point converges in distribution.
\end{abstract}

\section{Introduction}

In this paper, we continue our study on critical branching random walks. In the previous paper \cite{Z161}, we show that (under some assumptions about the offspring distribution and jump distribution), when $d\geq 5$, fix a finite subset $K\subseteq\Z^d$ and write $\Snake_x$ for the branching random walk starting from $x$, then,
\begin{equation}\label{p1}
\lim_{x\rightarrow \infty}\|x\|^{d-2}\cdot P(\Snake_x \text{ visits } K)= a_d\BCap(K);
\end{equation}
\begin{equation*}
\lim_{x\rightarrow \infty}P(\Snake_x(\tau_K)=a|\Snake_x \text{ visits } K)=\EsC_K(a)/\BCap(K),
\end{equation*}
\begin{equation*}
\lim_{x\rightarrow \infty}P(\Snake_x(\xi_K)=a|\Snake_x \text{ visits } K)=\Esc_K(a)/\BCap(K),
\end{equation*}
where $\tau_K$ and $\xi_K$ are the first and last visiting point (according to the Depth-First search on the tree) respectively, $a_d=\frac{1}{2d^{(d-2)/2}\sqrt{\det Q}}\Gamma(\frac{d-2}{2})\pi^{-d/2}$, $\|x\|=\sqrt{x\cdot Q^{-1}x}/\sqrt{d}$ with $Q$ being the covariance matrix of the jump distribution and $\EsC_K(a), \Esc_K(a)$ and $\BCap(K)$ are introduced in that paper.

In this paper we concentrate on the case $d=4$ and we provide the asymptotics of the probability of visiting a finite subset by a branching random walk and some related results. Before stating our results, let us specify the assumptions that will be in force throughout this paper. We always assume:

\begin{itemize}
\item the offspring distribution $\mu$ is a given distribution on $\N$ with mean one and finite variance $\sigma^2>0$;
\item the jump distribution $\theta$ is a given distribution on $\Z^4$ with mean zero not supported on a strict subgroup of $\Z^4$ and has finite exponential moment, i.e. for some $\lambda>0$,
    $$
    \sum_{z\in\Z^4}\theta(z)\cdot \exp(\lambda|z|)<\infty.
    $$
\end{itemize}

Our first result is
\begin{theorem}\label{MT4}
For any nonempty finite subset $K$ of $\Z^4$, we have:
\begin{equation}\label{MT_4}
\lim_{x\rightarrow \infty} (\|x\|^2\log \|x\|)\cdot P(\Snake_x\text{ visits }K)=\frac{1}{2\sigma^2}.
\end{equation}
\end{theorem}

Note that, compared with the analogous result \eqref{p1} in high dimensions, we have a logarithm correction in the limit and the right hand side is independent of $K$. Similar to the case in the high dimensions, we also show that conditioned on visiting, the first visiting point converges in distribution:
\begin{theorem}\label{MT6}
Assume further that $\theta$ has finite range. Then, for any nonempty finite subset $K$ of $\Z^4$ and $a\in K$, we have
\begin{equation}
\lim_{x\rightarrow \infty}P(\Snake_x(\tau_K)=a|\Snake_x \text{ visits } K)=\frac{\sigma^2}{4\pi^2\sqrt{\det Q}}\EE_K(a),
\end{equation}
where $\tau_K$ is the first visiting time, and $\EE_K(a)$ is defined later in \eqref{def-EE}.
\end{theorem}

In the supercritical dimensions $d\geq5$, conditioned on visiting, the 'entering measure' also converges in distribution (See Theorem 9.1 in \cite{Z161}). However, this is incorrect in the critical dimension $d=4$. The reason is that conditioned on visiting, the total number of 'entering points' will blow up, as the starting point goes to infinity. The theorem above states that conditioned on visiting , the first 'entering point' converges in distribution. By symmetry, one can see that the conditional last 'entering point' also converges to the same distribution. The distribution of the last visiting point is determined by the last 'entering point', hence converges.

Furthermore, we give the asymptotics of the range $R_n$ of the critical branching random walk conditioned on the total size of offsprings being $n$:
\begin{theorem}\label{MT5}
Assume further that $\theta$ is symmetric. Then, we have:
\begin{equation}\label{MT_5}
\frac{\log n}{n}R_n \stackrel{n\rightarrow \infty}{\longrightarrow} \frac{16\pi^2\sqrt{\det Q}}{\sigma^2}
\quad \text{in probability}.
\end{equation}
\end{theorem}

The case when $\mu$ is the critical geometric distribution was proved by Le Gall and Lin (\cite{LL141}):
\begin{equation*}
\frac{\log n}{n}R_n\stackrel{n\rightarrow \infty}{\longrightarrow} 8\pi^2\sqrt{\det Q}\quad \text{in }L^2.
\end{equation*}

\section{Preliminaries}
We begin with some notations. For a set $K\subseteq\Z^4$, we write $|K|$ for its cardinality. We write $K\ssubset\Z^4$ to express that $K$ is a finite nonempty subset of $\Z^4$. For $x\in\Z^4$, we denote by $|x|$ the Euclidean norm of $x$. We will mainly use the norm $\|\cdot\|$ corresponding to the jump distribution $\theta$, i.e. $\|x\|=\sqrt{x\cdot Q^{-1}x}/2$, where $Q$ is the covariance matrix of $\theta$. For convenience, we set $|0|=\|0\|=0.5$. We denote by $\diam(K)=\sup\{\| a-b\|:a, b\in K\}$, the diameter of $K$ and by $\Rad(K)=\sup\{\| a\| : a\in K\}$, the radius of $K$ respect to $0$. We write $\Ball(r)$ for the ball $\{z\in\Z^4: \| z\|\leq r\}$ and $\BallE(r)$ for the Euclidean ball $\{z\in\Z^4: |z|\leq r\}$. For any subsets $A,B$ of $\Z^4$, we denote by $\dist(A,B)=\inf\{ \| x-y \|: x\in A, y\in B\}$ the distance between $A$ and $B$. When $A=\{x\}$ consists of just one point, we just write $\dist(x,B)$ instead.
For any path $\gamma:\{0,\dots,k\}\rightarrow \Z^4$, we let $|\gamma|$ stand for $k$, the length, i.e. the number of edges of $\gamma$ , $\widehat{\gamma}$ for $\gamma(k)$, the endpoint of $\gamma$ and $[\gamma]$ for $k+1$, the number of vertices of $\gamma$. Sometimes we just use a sequence of vertices to express a path. For example, we may write $(\gamma(0),\gamma(1),\dots, \gamma(k))$ for the path $\gamma$. For any $B\subset \Z^d$, we write $\gamma\subseteq B$ to express that all vertices of $\gamma$ except the starting point and the endpoint, lie inside $B$, i.e. $\gamma(i)\in B$ for any $1\leq i\leq |\gamma|-1$. If the endpoint of a path $\gamma_1:\{0,\dots,|\gamma_1|\}\rightarrow \Z^4$ coincides with the starting point of another path $\gamma_2:\{0,\dots,|\gamma_2|\}\rightarrow \Z^4$, then we can define the composite of $\gamma_1$ and $\gamma_2$ by concatenating $\gamma_1$ and $\gamma_2$:
$$
\gamma_1\circ\gamma_2:\{0,\dots,|\gamma_1|+|\gamma_2|\}\rightarrow \Z^4,
$$
$$
\gamma_1\circ\gamma_2(i)=\left\{
\begin{array}{ll}
\gamma_1(i),&\text{ for } i\leq|\gamma_1|;\\
\gamma_2(i-|\gamma_1|),&\text{ for } i\geq|\gamma_1|.\\
\end{array}
\right.
$$

We now state our convention regarding constants. Throughout the text (unless otherwise specified), we use $C$ and $c$ to denote positive constants depending only on the critical distribution $\mu$ and the jump distribution $\theta$, which may change from place to place. Dependence of constants on additional parameters will be made or stated explicit. For example, $C(\lambda)$ stands for a positive constant depending on $\mu,\theta, \lambda$. For functions $f(x)$ and $g(x)$, we write $f\sim g$ if $\lim_{x\rightarrow \infty}(f(x)/g(x))=1$. We write $f\preceq g$, respectively $f\succeq g$, if there exist constants $C$ such that, $f\leq Cg$, respectively $f\geq Cg$. We use $f\asymp g$ to express that $f\preceq g$ and $f\succeq g$. We write $f\ll g$ if $\lim_{x\rightarrow \infty}(f(x)/g(x))=0$.

\subsection{Finite and infinite trees.} We are interested in rooted ordered trees (plane trees), in particular, Galton-Watson (GW) trees and its companions. Recall that $\mu=(\mu(i))_{i\in\N}$ is a given critical distribution with finite variance $\sigma^2>0$. Note that we exclude the trivial case that $\mu(1)=1$. Throughout this paper, $\mu$ will be fixed. Define another probability measure $\widetilde{\mu}$ on $\N$, call the \textbf{adjoint measure} of $\mu$ by setting $\widetilde{\mu}(i)=\sum_{j=i+1}^\infty \mu(j)$. Since $\mu$ has mean $1$, $\widetilde{\mu}$ is indeed a probability measure. The mean of $\widetilde{\mu}$ is $\sigma^2/2$. A Galton-Watson process with distribution $\mu$ is a process starting with one initial particle, with each particle having independently a random number of children due to $\mu$. The Galton-Watson tree is just the family tree of the Galton-Watson process, rooted at the initial particle. We simply write \text{$\mu$-GW tree} for the Galton-Watson tree with offspring distribution $\mu$.  If we just change the law of the number of children for the root, using $\widetilde{\mu}$ instead of $\mu$ (for other particles still use $\mu$), the new tree is called an \textbf{adjoint $\mu$-GW tree}. The \textbf{infinite $\mu$-GW tree} is constructed in the following way: start with a semi-infinite line of vertices, called the spine, and graft to the left of each vertex in the spine an independent adjoint $\mu$-GW tree, called a bush. The infinite $\mu$-GW tree is rooted at the first vertex of the spine. Here the left means that we assume every vertex in spine except the root is the youngest child (the latest in the Depth-First search order) of its parent. The \textbf{invariant $\mu$-GW tree} is defined similarly to the infinite $\mu$-GW tree except that we graft to the root (the first vertex of the spine) an independent $\mu$-GW tree instead of an independent adjoint $\mu$-GW tree.  We also need to introduce the so-called \textbf{$\mu$-GW tree conditioned on survival}. Start with a semi-infinite path, called the spine, rooted at the starting point. For each vertex in the spine, with probability $\mu(i+j+1)$ ($i,j\in\N$), it has totally $i+j+1$ children, with exactly $i$ children elder than the child corresponding to the next vertex in the spine, and exactly $j$ children younger. For any vertex not in the spine, it has a random number of children due to $\mu$. The number of children for different vertices are independent. The random tree generated in this way is just the $\mu$-GW tree conditioned on survival. Each tree is ordered using the classical order according to Depth-First search starting from the root. Note that the subtree generated by the vertices of the spine and all vertices on the left of the spine of the $\mu$-GW tree conditioned on survival has the same distribution as the infinite $\mu$-GW tree.

\subsection{Tree-indexed random walk.} Now we introduce the random walk in $\Z^4$ with jump distribution $\theta$, indexed by a random plane tree $T$. First choose some $a\in\Z^4$ as the starting point. Conditionally on $T$ we assign independently to each edge of $T$ a random variable in $\Z^4$ according to $\theta$. Then we can uniquely define a function $\mathcal{S}_T: T\rightarrow \Z^4$, such that, for every vertex $v \in T$ (we also use $T$ for the set of all vertices of the tree $T$), $\mathcal{S}_T(v)-a$ is the sum of the variables of all edges belonging to the unique simple path from the root $o$ to the vertex $u$ (hence $\mathcal{S}_T(o)=a$). A plane tree $T$ together with this random function $\mathcal{S}_T$ is called $T$-indexed random walk starting from $a$. When $T$ is a $\mu$-GW tree, an adjoint $\mu$-GW tree, and an infinite $\mu$-GW tree respectively, we simply call the tree-indexed random walk a \textbf{snake}, an \textbf{adjoint snake} and an \textbf{infinite snake} respectively. We write $\Snake_x$, $\Snake'_x$ and $\Snake^{\infty}_x$ for a snake, an adjoint snake, and an infinite snake, respectively, starting from $x\in \Z^4$. Note that a snake is just the branching random walk with offspring distribution $\mu$ and jump distribution $\theta$. We also need to introduce the \textbf{reversed infinite snake} starting from $x$, $\Snake^{-}_x$, which is constructed in the same way as $\Snake^{\infty}_x$ except that the variables assigned to the edges in the spine are now due to not $\theta$ but the reverse distribution $\Rtheta$ of $\theta$ (i.e. $\Rtheta(x)\DeFine\theta(-x)$ for $x\in \Z^4$) and similarly the \textbf{invariant snake} starting from $x$, $\Snake^{I}_x$, which is constructed by using the invariant $\mu$-GW tree as the random tree $T$ and using $\Rtheta$ for all edges of the spine of $T$ and $\theta$ for all other edges. For an infinite snake (or reversed infinite snake, invariant snake), the random walk indexed by its spine, called its backbone, is just a random walk with jump distribution $\theta$ (or $\Rtheta$).  Note that all snakes here certainly depend on $\mu$ and $\theta$. Since $\mu$ and $\theta$ are fixed throughout this paper, we omit their dependence in the notation.

Notation: we write $\pS_A(x)$ ($\rS_A(x)$) for the probability that $\Snake_x$ ($\Snake'_x$) visits $A$.

\subsection{Random walk with killing.} We will use the tools of random walk with killing. Suppose that when the random walk is currently at position $x\in\Z^4$, then it is killed, i.e. jumps to a 'cemetery' state $\delta$, with probability $\KRW(x)$, where $\KRW:\Z^4\rightarrow [0,1]$ is a given function. In other words, the random walk with killing rate $\KRW(x)$ (and jump distribution $\theta$) is a Markov chain $\{X_n:n\geq0\}$ on $\Z^4\cup\{\delta\}$ with transition probabilities $p(\cdot,\cdot)$ given by: for $x,y\in\Z^4$,
\begin{equation*}
p(x,\delta)=\KRW(x), \quad p(\delta, \delta)=1, \quad p(x,y)=(1-\KRW(x))\theta(y-x).
\end{equation*}
For any path $\gamma:\{0,\dots, n\}\rightarrow \Z^4$ with length $n$, its probability weight $\BRW(\gamma)$ is defined to be the probability that the path consisting of the first $n$ steps for the random walk with killing starting from $\gamma(0)$ is $\gamma$. Equivalently,
\begin{equation}\label{def-b}
\BRW(\gamma)=\prod_{i=0}^{|\gamma|-1}(1-\KRW(\gamma(i)))\theta(\gamma(i+1)-\gamma(i))
=\SRW(\gamma)\prod_{i=0}^{|\gamma|-1}(1-\KRW(\gamma(i))),
\end{equation}
where $\SRW(\gamma)=\prod_{i=0}^{|\gamma|-1}\theta(\gamma(i+1)-\gamma(i))$ is the probability weight of $\gamma$ corresponding to the random walk with jump distribution $\theta$. In this paper, the killing function is always of the form $\KRW(x)=\KRW_A(x)=\rS_A(x)$ for some finite subset $A\subseteq\Z^d$ (unless otherwise specified).

Now we can define the corresponding Green function for $x,y\in \Z^4$:
\begin{equation*}
G_A(x,y)=\sum_{n=0}^{\infty}P(S^{\KRW}_x(n)=y)=\sum_{\gamma:x\rightarrow y}\BRW(\gamma).
\end{equation*}
where $S^{\KRW}_x=(S^{\KRW}_x(n))_{n\in \N}$ is the random walk (with jump distribution $\theta$) starting from $x$, with killing function $\KRW_A$, and the last sum is over all paths from $x$ to $y$. For $x\in \Z^4, A,K\subset \Z^4$, we write $G_A(x,K)$ for $\sum_{y\in K}G_A(x,y)$.

For any $B\subset \Z^4$ and $x, y\in \Z^4$, define the harmonic measure:
$$
\Hm^{B}_A(x,y)= \sum_{\gamma:x\rightarrow y, \gamma\subseteq B}\BRW(\gamma).
$$

When the killing function $\KRW \equiv 0$, the random walk with this killing is just random walk without killing and we write $\Hm^{B}(x,y)$ for this case.

We will repeatedly use the following First-Visit Lemma. The idea is to decompose a path according to the first or last visit of a set.
\begin{lemma} \label{bd-1}
For any $B\subseteq \Z^4$ and $a\in B, b\notin B$, we have:
$$
G_A(a,b)=\sum_{z\in B^c}\Hm^{B}_A(a,z)G_\KRW(z,b)=\sum_{z\in B}G_A(a,z)\Hm^{B^c}_A(z,b);
$$
$$
G_A(b,a)=\sum_{z\in B}\Hm^{B^c}_A(b,z)G_\KRW(z,a)=\sum_{z\in  B^c}G_A(b,z)\Hm^{B}_A(z,a).
$$
\end{lemma}

Since our jump distribution $\theta$ may be unbounded, we need the following Overshoot Lemma:
\begin{lemma}\label{overshoot}
For any $r,s>1$ and $p\in \N^+$, there exists some $C(p)>0$, such that, for any $a\in \Ball(r)$, we have:
\begin{equation}\label{os}
\sum_{y\in(\Ball (r+s))^c} \Hm^{B}(a,y)\leq C(p)\frac{r^2}{s^{p}},\quad
\sum_{y\in(\Ball (r+s))^c} \Hm^{B}(y,a)\leq C(p)\frac{r^2}{s^{p}}.
\end{equation}
\end{lemma}

\subsection{Some facts about random walk and the Green function.}
For $x\in\Z^4$, we write $S_x=(S_x(n))_{n\in\N}$ for the random walk with jump distribution $\theta$ starting from $S_x(0)=x$. Recall that the norm $\| \cdot \|$ corresponding to $\theta$ for every $x\in \Z^4$ is defined to be $\|x\|=\sqrt{x\cdot Q^{-1}x}/2$, where $Q$ is the covariance matrix of $\theta$. Note that $\|x\|\asymp |x|$, especially, there exists $c>1$, such that $\Ball(c^{-1}n)\subseteq\BallE(n)\subseteq\Ball(cn)$, for any $n\geq1$. The Green function $g(x,y)$ is defined to be:
$$
g(x,y)=\sum_{n=0}^{\infty}P(S_x(n)=y)=\sum_{\gamma:x\rightarrow y}\SRW(\gamma).
$$
We write $g(x)$ for $g(0,x)$.

Our assumptions about the jump distribution $\theta$ guarantee the standard estimate for the Green function (see e.g. Theorem 4.3.5 in \cite{LL10}):
\begin{equation}\label{green}
g(x)\sim a_4\| x \|^{-2};
\end{equation}
where $a_4=1/(8\pi^2\sqrt{\det Q})$ with $Q$ being the covariant matrix of $\theta$.

By LCLT, one can get the following lemma.
\begin{lemma}\label{pro-g2}
\begin{equation}
\lim_{n\rightarrow\infty}sup_{x\in\Z^4}\left(\sum_{\gamma:0\rightarrow x,|\gamma|\geq n|x|^2}
\SRW(\gamma)/g(x)\right)=0.
\end{equation}
\end{lemma}

\subsection{Some results in the previous papers.}
We first recall some notations from \cite{Z161}. Fix a finite subset $K\ssubset \Z^4$. For any $x\in \Z^4$, define the position-dependent distribution $\mu_x=\mu_{x,K}$ on $\N$ by:
\begin{equation}\label{mu-x}
\mu_x(m)=\left\{
\begin{array}{ll}
\sum_{l\geq0}\mu(l+m+1)(\rp(x))^l/(1-\rS_K(x)),& \text{when }x\notin K;\\
\mu(m),& \text{when }x\in K; \\
\end{array}
\right.
\end{equation}
where $\rp(x)$ is the probability that a snake starting from $x$ does not visit $K$ conditioned on the initial particle having only one child. Write $N(x)=N_K(x)=E\mu_x$.
Note that
$$
\lim_{x\rightarrow \infty}N(x)=E\tilde{\mu}=\frac{\sigma^2}{2}.
$$

When a snake $\Snake_x=(T,\mathcal{S}_T)$ visits $K$, since $T$ is an ordered tree, we have the unique first vertex, denoted by $\tau_K$, in $\{v\in T: \mathcal{S}_T(v)\in K\}$ due to the default order. We say $\mathcal{S}_T(\tau_K)$ is the visiting point or $\Snake_x$ visits $K$ at $\mathcal{S}_T(\tau_K)$. Let $(v_0,v_1,\dots,v_k)$ be the unique simple path in $T$ from the root $o$ to $\tau_K$ and say $\Snake_x$ visits $K$ via $\gamma$. Let $\widetilde{b}_i$ be the number of the younger  brothers of $v_i$, for $i=1,\dots,k$.
The key observation is the following proposition (see Sec 5 and Sec 9 in \cite{Z161}):
\begin{prop}\label{KKey}
For any $\gamma=(\gamma(0),\dots,\gamma(k))\subseteq K^c$ starting from $x$, ending at $K$, we have:
\begin{equation}\label{key}
P(\Snake_x \text{ visits }K\text{ via }\gamma)=\BRW(\gamma);
\end{equation}
\begin{equation}\label{p2}
\pS_K(x)=\sum_{\gamma:x\rightarrow K}\BRW(\gamma)=G_K(x,K);
\end{equation}
Furthermore, conditioned on the event $\Snake_x$ visits $K$ via $\gamma$, $\widetilde{b}_i$($i=1,\dots,k$) are independent, with the following distribution:
$$
\widetilde{b}_i \stackrel{d}{=}\mu_{\gamma(i-1)}.
$$
\end{prop}

For $\rS_K(x)$, we have (see (8.6) in \cite{Z161})
\begin{equation}\label{r-p}
\rS_K(x)\asymp \pS_K(x).
\end{equation}

Note that in \cite{Z161} we address the supercritical dimensions ($d\geq5$). However, for the results above, we do need the assumption $d\geq5$. They hold for any $d\in \N^+$.

We also need the following result from \cite{Z15}
\begin{equation}\label{previous}
\pS_{\{0\}}(x)\asymp (\|x\|^2\log \|x\|)^{-1}.
\end{equation}

\section{The visiting probability.}
The main goal of this section is to prove Theorem \ref{MT4}. In this section and the next, we fix a finite nonempty subset $K\ssubset \Z^4$ and therefore the constants may also depends on $K$.

The first step is to construct the following estimate of the Green function:
\begin{lemma}\label{con-G-local}
For any $\alpha\in (0,1/2)$,  we have:
\begin{equation}\label{G_local}
\lim_{x,y\rightarrow \infty:\|x\|/(\log\|x\|)^\alpha\leq\|y\|\leq\|x\|\cdot(\log\|x\|)^\alpha}
\frac{G_K(x,y)}{g(x,y)}=1.
\end{equation}
\end{lemma}
\begin{remark}
In supercritical dimensions, we have $G_K(x,y)\sim g(x,y)$ (see Lemma 6.1 in \cite{Z161}). In the critical dimension, this holds only when $x,y$ are not too far away from each other, compared with their norms. We will give a more precise asymptotic behavior of $G_K$ in next section.
\end{remark}

\begin{proof}[Proof of Lemma \ref{con-G-local}]
We use the same idea in the proof for a similar form in the supercritical dimension(see Section 6 of \cite{Z161}). Since $\alpha<1/2$, we can pick up some $\beta,\epsilon>0$, such that $\epsilon+2\alpha+2\beta<1$. Without loss of generality, we assume $\|y\|\geq \|x\|$. Let $r=\|x\|/\log^\beta \|x\|$ and
$$
\Gamma_1=\{\gamma:x\rightarrow y|\;|\gamma|\geq (\log \|x\|)^\epsilon\|x-y\|^2\};
$$
$$
\Gamma_2=\{\gamma:x\rightarrow y|\gamma \text{ visits }\Ball(r)\}.
$$
We just need to check: (when $\|x\|\rightarrow\infty$)
$$
\sum_{\gamma\in\Gamma_1}\SRW(\gamma)/g(x,y)\rightarrow0;\quad \sum_{\gamma\in\Gamma_2}\SRW(\gamma)/g(x,y)\rightarrow0;
$$
$$
\BRW(\gamma)/\SRW(\gamma)\rightarrow 1, \text{ for any }\gamma:x\rightarrow y, \notin \Gamma_1\cup\Gamma_2.
$$
The first one follows from Lemma \ref{pro-g2}. The second one can be obtained by: (let $B=\Ball(r)$):
\begin{align*}
\sum_{\gamma\in \Gamma_2}\SRW(\gamma)=&\sum_{a\in B}\Hm^{B^c}(x,a)
g(a,y)\asymp \sum_{a\in B}\Hm^{B^c}(x,a)\|y\|^{-2}\\
=&P(S_x \text{ visits } B)\cdot \|y\|^{-2}\asymp  (r/\|x\|)^{2} \|y\|^{-2}\\
\preceq &  (\log \|x\|)^{2\beta}\|x-y\|^{-2} \ll g(x,y).\\
\end{align*}
Note that the estimate of $P(S_x \text{ visits } \Ball(r))\asymp (r/\|x\|)^{2}$ is standard, and for the second last inequality we use $\|y\|\geq(\|x\|+\|y\|)/2\succeq\|x-y\|$.

For the third one, note that in the critical dimension $d=4$, by \eqref{previous} and \eqref{r-p}, the killing function $\KRW(z)=\rS_K(z)\asymp\pS_K(z)\preceq1/(\|z\|^2\log\|z\|)$. Hence, we have: \\
for any $\gamma:x\rightarrow y, \notin \Gamma_1\cup\Gamma_2$,
\begin{align*}
\BRW(\gamma)/\SRW(\gamma)=&\prod_{i=0}^{|\gamma|-1}\left(1-\KRW(\gamma(i))\right)
\geq(1-c/(r^{2}\log r))^{|\gamma|}
\geq 1-c|\gamma|/(r^{2}\log r)\\
\geq &1-c(\log \|x\|)^\epsilon\|y\|^2/((\|x\|/\log^\beta \|x\|)^2(\log \|x\|)) \\
\geq &1-c(\log \|x\|)^\epsilon(\|x\|\log^\alpha\|x\|)^2/((\|x\|/\log^\beta \|x\|)^2(\log \|x\|)) \\
\geq &1-c(\log\|x\|)^{\epsilon+2\alpha+2\beta}/\log\|x\|\rightarrow 1.
\end{align*}
\end{proof}

Let $N$ be the number of times of visiting $K$. We need to estimate $E(N|\Snake_x \text{ visits }K \\
\text{ via }\gamma)$. For any finite path $\gamma$, define
\begin{equation}\label{k1}
\Num(\gamma)=\sum_{i=0}^{|\gamma|}N(\gamma(i))g(\gamma(i),K);\quad
\Num^-(\gamma)=\sum_{i=0}^{|\gamma|-1}N(\gamma(i))g(\gamma(i),K).
\end{equation}

By Proposition \ref{KKey}, we have:
$$
E(N|\Snake_x \text{ visits }K \text{ via }\gamma))=\Num(\gamma).
$$

Hence, we have:
\begin{equation}\label{eq-0}
\sum_{\gamma:x\rightarrow K}\BRW(\gamma)\Num(\gamma)=g(x,K)\sim a_4|K|\|x\|^{-2}.
\end{equation}

The main step is to control the sum of the escape probabilities:
\begin{prop}\label{thm-es}
\begin{equation}\label{thm_es}
\sum_{\gamma:\Ball(2n)\setminus \Ball(n)\rightarrow:K,\gamma\subseteq \Ball(n)}\BRW(\gamma)
\sim \frac{4\pi^2\sqrt{\det Q}}{\sigma^2}\frac{1}{\log n}.
\end{equation}
\end{prop}
In order to prove this proposition, we need two lemmas about random walks. They are adjusted versions of Lemma 17 and Lemma 18 in \cite{LL141}. Let $(S_j)_{j\in \N}$ be the random walk (starting from $0$), $\tau_n$ be the first visiting time of $\Ball(n)^c$ by the random walk and $h(x):\Z^4\rightarrow \R^+$ is a fixed positive function satisfying $h(x)\sim a_4\|x\|^{-2}$.

\begin{lemma}\label{yy1}
For $p=1,2$, there exists a constant $C(p)$ (also depending on $h$) such that, for every $n\geq 2$,
\begin{equation}
E(\sum_{j=0}^{\tau_n}h(S_j))^p\leq C(p)(\log n)^p.
\end{equation}.
\end{lemma}

\begin{lemma}\label{yy2}
For every $\alpha,p>0$, there exists a constant $C_\alpha(p)$ (also depending on $h$) such that, for every $n\geq 2$, we have
\begin{equation}
P(|\sum_{k=0}^{\tau_n}h(S_j)-4a_4\log n|\geq \alpha \log n)\leq C_\alpha(p)(\log n)^{-p}.
\end{equation}
\end{lemma}

In fact, we apply both lemmas for the reversed random walk (that is, with jump distribution $\theta^-$) other than the original random walk. Let $h(x)=2N(x)g(x,K)/(\sigma^2|K|)$.  Recall that $N(x)=E\mu_x\sim \sigma^2/2$ and $\Num(\gamma)=\sum \sigma^2|K|h(\gamma(i))/2$). Hence, for any $a\in K$ we have:
$$
\sum_{\gamma:\Ball(n)^c\rightarrow a, \gamma\subseteq\Ball(n)}\SRW(\gamma)(\Num(\gamma))^2\preceq (\log n)^2,
$$
$$
\sum_{\gamma:\Ball(n)^c\rightarrow a, \gamma\subseteq\Ball(n), |\Num(\gamma)-2a_4|K|\sigma^2\log n|\geq \alpha\log n}\SRW(\gamma)\leq C_\alpha(p)(\log n)^{-p}.
$$
By monotonicity and summation, we get:
\begin{equation}\label{yy_1}
\sum_{\gamma:\Ball(n)^c\rightarrow K,\gamma\subseteq\Ball(n)\setminus K}\SRW(\gamma)(\Num(\gamma))^2\preceq (\log n)^2,
\end{equation}
\begin{equation}\label{yy_2}
\sum_{\gamma:\Ball(n)^c\rightarrow K,\gamma\subseteq\Ball(n)\setminus K,|\Num(\gamma)-2a_4|K|\sigma^2\log n|\geq \alpha\log n}\SRW(\gamma)\preceq C_\alpha(p)(\log n)^{-p}.
\end{equation}

Let us make some comments about the proofs. Lemma 18 in \cite{LL14} states that
$$
P(|\sum_{k=0}^{n}g(S_j)-2a_4\log n|\geq \alpha \log n)\leq C_\alpha(\log n)^{-3/2}.
$$
where $g$ is the Green function. Their argument is to derive an analogous result for Brownian motion and then to transfer this result to the random walk via the strong invariance principle. This argument also works here with small adjustments. Note that it is assumed there that the jump distribution $\theta$ is symmetric (besides having exponential tail). However if one checks the proof there, one can see that the assumption of symmetry is not needed and $g(x)$ can be replaced by any $h(x)$ satisfying $h(x)\sim a_4\|x\|^{-2}$ . Moreover, the exponent $3/2$ can be replaced by any positive constant $p$ with minor modifications.
Combing this with the fact that for any fixed $\epsilon>0$, $P(\tau_n\notin [n^{2-\epsilon}, n^{2+\epsilon}])=o((\log n)^{-p})$, one can get Lemma \ref{yy2}.

For Lemma \ref{yy1}, we give a direct proof here:
\begin{proof}[Proof of Lemma \ref{yy1}]

\textbf{For p=1},
\begin{align*}
E(\sum_{j=0}^{\tau_n}h(S_j))&\preceq\sum_{z\in\Ball(n)}h(z)E(\sum_{j=0}^{\tau_n}\mathbf{1}_{S_j=z})
\preceq\sum_{z\in\Ball(n)}|z|^{-2}E(\sum_{j=0}^{\infty}\mathbf{1}_{S_j=z})\\
&=\sum_{z\in\Ball(n)}|z|^{-2}g(0,z)\asymp\sum_{z\in\Ball(n)}|z|^{-4}\asymp\log n.
\end{align*}
\textbf{For p=2},
\begin{align*}
&E(\sum_{j=0}^{\tau_n}h(S_j))^2\preceq E(\sum_{z\in\Ball(n)}h(z)\sum_{j=0}^{\tau_n}\mathbf{1}_{S_j=z})^2
\preceq E(\sum_{z\in\Ball(n)}|z|^{-2}\sum_{j=0}^{\infty}\mathbf{1}_{S_j=z})^2\\
&=\sum_{z,w\in\Ball(n)}|z|^{-2}|w|^{-2}E(\sum_{j=0}^{\infty}\mathbf{1}_{S_j=z}\sum_{i=0}^{\infty}\mathbf{1}_{S_i=w}).
\end{align*}
Write $A_x=\sum_{j=0}^{\infty}\mathbf{1}_{S_j=x}$ and $A=A_z+A_w$. We point out that
\begin{equation}\label{t1}
E(A_zA_w)\preceq (|z|^{-2}+|w|^{-2})|z-w|^{-2}.
\end{equation}

If so, note that
\begin{align*}
&\sum_{z,w\in\Ball(n)}|z|^{-2}|w|^{-2}(|z|^{-2}+|w|^{-2})|z-w|^{-2}
\preceq\sum_{z,w\in\Ball(n):|z|\leq |w|}|z|^{-4}|w|^{-2}|z-w|^{-2}\\
\leq&\sum_{w\in\Ball(n)}(\sum_{z:|z|\leq|w|,|z-w|\geq |w|/2}+\sum_{z:|z|\leq|w|,|z-w|\leq |w|/2})|z|^{-4}|w|^{-2}|z-w|^{-2}\\
\preceq&\sum_{w\in\Ball(n)}
(\sum_{z:|z|\leq|w|,|z-w|\geq |w|/2}|z|^{-4}|w|^{-4}+\sum_{z:|z|\leq|w|,|z-w|\leq |w|/2}|w|^{-6}|z-w|^{-2})\\
\preceq&\sum_{w\in\Ball(n)}((\log|w|)|w|^{-4}+|w|^{-4})\asymp(\log n)^2.
\end{align*}
and then one can get $E(\sum_{j=0}^{\tau_n}h(S_j))^2\preceq (\log n)^2$.
We now only need to show \eqref{t1}. Without loss of generality, assume $z\neq w$ (the case $z=w$ can be addressed similarly with small adjustments).
Note that
$$
E(A_zA_w)\leq E(A^2;A_z>0,A_w>0)\asymp\sum_{k\geq2}kP(A\geq k,A_z>0,A_w>0).
$$
By Markov property, one can see that:
\begin{multline*}
P(A\geq k,A_z>0,A_w>0)\leq \\
P(A_z>0)((k-1)P_z(A_w>0)c^{k-2})+P(A_w>0)((k-1)P_w(A_z>0)c^{k-2}),
\end{multline*}
where we write $P_x$ for the law of random walk starting from $x$ and
$$
c=\sup_{x\neq y\in\Z^4}{P_x(A_x+A_y>1)}<1.
$$
Hence, we have:
\begin{align*}
&\sum_{k\geq2}kP(A\geq k,A_z>0,A_w>0)\\
\leq&(\sum_{k\geq2}k(k-1)c^{k-2})(P(A_z>0)P_z(A_w>0)+P(A_w>0)P_w(A_z>0))\\
\preceq& P(A_z>0)P_z(A_w>0)+P(A_w>0)P_w(A_z>0)\\
\asymp& (|z|^{-2}+|w|^{-2})|z-w|^{-2}.
\end{align*}

\end{proof}

\begin{proof}[Proof of Proposition \ref{thm-es}]
We first show the following weaker result:
\begin{lemma}
\begin{equation}\label{lem-1}
\sum_{\gamma:\Ball(2n)\setminus \Ball(n)\rightarrow:K,\gamma\subseteq \Ball(n)}\BRW(\gamma)\preceq (\log n)^{-1},
\quad\text{as }n\rightarrow \infty.
\end{equation}
\end{lemma}
\begin{proof}
By \eqref{previous} and \eqref{key}, we have:
\begin{equation}\label{known}
\sum_{\gamma:x\rightarrow K}\BRW(\gamma)\preceq (\|x\|^2\log \|x\|)^{-1}.
\end{equation}
Pick some $x\in\Z^d$ such that $n=\lfloor \|x\|(\log\|x\|)^{-1/4}\rfloor$. Let $B=\Ball(n)$ and $B_1=\Ball(2n)$. By the First-Visit Lemma, we have:
\begin{align*}
\sum_{\gamma:x\rightarrow K}\BRW(\gamma)&=\sum_{a\in K}\sum_{z\in B^c}G_K(x,z)\Hm_\KRW^B(z,a)\geq
\sum_{a\in K}\sum_{z\in B_1\setminus B}G_K(x,z)\Hm_\KRW^B(z,a)\\
&\stackrel{\eqref{G_local}}{\succeq}\sum_{a\in K}\sum_{z\in B_1\setminus B}g(x,z)\Hm_\KRW^B(z,a)\succeq\|x\|^{-2}\sum_{\gamma:\Ball(2n)\setminus \Ball(n)\rightarrow K,\gamma\subseteq \Ball(n)}\BRW(\gamma).\\
\end{align*}
Combining this with \eqref{known} gives \eqref{lem-1}.
\end{proof}
We need to transfer \eqref{eq-0} to the following form:
\begin{lemma}
\begin{equation}\label{eq-1}
\lim_{n\rightarrow \infty}\sum_{\gamma:\Ball(2n)\setminus \Ball(n)\rightarrow K,\gamma\subseteq \Ball(n)}\Num(\gamma)\BRW(\gamma)= |K|.
\end{equation}
\end{lemma}
\begin{proof}[Proof of \eqref{eq-1}]
Pick some $x\in\Z^d$ such that $n=\lfloor \|x\|(\log\|x\|)^{-1/4}\rfloor$. Let $B=\Ball(n)$ and $B_1=\Ball(2n)$. By decomposing $\gamma$ at the last step in $B$, one can get:
\begin{align*}
&\sum_{\gamma:x\rightarrow K}\BRW(\gamma)\Num(\gamma)=
\sum_{z\in B^c}\sum_{\gamma_1:x\rightarrow z}\sum_{\gamma_2:z\rightarrow K,\gamma_2\subseteq B}
\BRW(\gamma_1)\BRW(\gamma_2)(\Num^-(\gamma_1)+\Num(\gamma_2))=\\
&\sum_{z\in B^c}\left(\sum_{\gamma_1:x\rightarrow z}\sum_{\gamma_2:z\rightarrow K,\gamma_2\subseteq B}
\BRW(\gamma_1)\BRW(\gamma_2)\Num^-(\gamma_1)+\sum_{\gamma_1:x\rightarrow z}\sum_{\gamma_2:z\rightarrow K,\gamma_2\subseteq B}
\BRW(\gamma_1)\BRW(\gamma_2)\Num(\gamma_2)\right)\\
&=\sum_{z\in B^c}\sum_{\gamma_2:z\rightarrow K,\gamma_2\subseteq B}\BRW(\gamma_2)\sum_{\gamma_1:x\rightarrow z}\BRW(\gamma_1)
\Num^-(\gamma_1)+\\
&\quad\quad\quad\quad\quad\quad\quad\quad\quad\quad\quad\quad\quad\quad\quad
\sum_{z\in B^c}\sum_{\gamma_2:z\rightarrow K,\gamma_2\subseteq B}\BRW(\gamma_2)\Num(\gamma_2)
\sum_{\gamma_1:x\rightarrow z}
\BRW(\gamma_1).
\end{align*}

We argue that the first term is negligible:
\begin{equation}\label{temp-0}
\sum_{z\in B^c}\sum_{\gamma_2:z\rightarrow K,\gamma_2\subseteq B}\BRW(\gamma_2)\sum_{\gamma_1:x\rightarrow z}\BRW(\gamma_1)
\Num^-(\gamma_1)\ll \|x\|^{-2}.
\end{equation}
Note that
\begin{align*}
\sum_{\gamma:x\rightarrow z}&\BRW(\gamma)\Num^-(\gamma)
\leq\sum_{w\in \Z^d}N(w)g(w,K)\sum_{\gamma:x\rightarrow z}
\BRW(\gamma)\sum_{i=0}^{|\gamma|}\mathbf{1}_{\gamma(i)=w}\\
&\preceq \sum_{w\in \Z^d}|w|^{-2}\sum_{\gamma:x\rightarrow w}\BRW(\gamma)
\sum_{\gamma:w\rightarrow z}\BRW(\gamma)
\leq\sum_{w\in \Z^d}|w|^{-2}g(x,w)g(w,z).\\
\end{align*}
In order to estimate the term above, we need the following easy lemma whose proof we postpone
\begin{lemma}\label{conv-3}
For any $a,b,c\in \Z^4$, we have:
\begin{equation}
\sum_{z\in\Z^4}|z-a|^{-2}|z-b|^{-2}|z-c|^{-2}\preceq \frac{1\vee\log(M/m)}{M^2},
\end{equation}
where $M=\max\{|a-b|,|b-c|,|c-a|\}$ and $m=\min\{|a-b|,|b-c|,|c-a|\}$.
\end{lemma}
By this lemma, when $z\in B_1\setminus B$,
$\sum_{\gamma:x\rightarrow z}\BRW(\gamma)\Num^-(\gamma)\preceq\frac{\log (\|x\|/n)}{\|x\|^2}$. Together with \eqref{lem-1}, we have
$$
\sum_{z\in B_1\setminus B}\sum_{\gamma_2:z\rightarrow K,\gamma_2\subseteq B}\BRW(\gamma_2)\sum_{\gamma_1:x\rightarrow z}\BRW(\gamma_1)
\Num^-(\gamma_1)\ll\|x\|^{-2}.
$$
Also by Lemma \ref{conv-3} when $z\in B^c_1$, $\sum_{\gamma:x\rightarrow z}\BRW(\gamma)\Num^-(\gamma)\preceq\frac{\log (\|x\|)}{\|x\|^2}$.
On the other hand, by the Overshoot Lemma, we have $\sum_{z\in B^c_1}\sum_{\gamma_2:z\rightarrow K,\gamma_2\subseteq B}\BRW(\gamma_2)\preceq n^{-4}$.
Hence,
$$
\sum_{z\in B^c_1}\sum_{\gamma_2:z\rightarrow K,\gamma_2\subseteq B}\BRW(\gamma_2)\sum_{\gamma_1:x\rightarrow z}\BRW(\gamma_1)\Num^-(\gamma_1)\ll\|x\|^{-2}.
$$
This completes the proof of \eqref{temp-0}. Combining \eqref{temp-0} with \eqref{eq-0} gives:
$$
\sum_{z\in B^c}\sum_{\gamma_2:z\rightarrow K,\gamma_2\subseteq B}\BRW(\gamma_2)\Num(\gamma_2)
\sum_{\gamma_1:x\rightarrow z}\BRW(\gamma_1)\sim a_4|K|\|x\|^{-2}.
$$
Now we aim to show
\begin{equation}
\sum_{z\in B_1^c}\sum_{\gamma_2:z\rightarrow K,\gamma_2\subseteq B}\BRW(\gamma_2)\Num(\gamma_2)
\sum_{\gamma_1:x\rightarrow z}\BRW(\gamma_1)\ll a_4|K|\|x\|^{-2}.
\end{equation}
If so, then we have:
\begin{equation}\label{o4}
\sum_{z\in B_1\setminus B}\sum_{\gamma_2:z\rightarrow K,\gamma_2\subseteq B}\BRW(\gamma_2)\Num(\gamma_2)
\sum_{\gamma_1:x\rightarrow z}\BRW(\gamma_1)\sim a_4|K|\|x\|^{-2}
\end{equation}
and combining this with Lemma \ref{con-G-local} gives \eqref{eq-1}.
Since $\sum_{\gamma_1:x\rightarrow z}\BRW(\gamma_1)=G_K(x,z)\preceq1$ and $\BRW(\gamma)\leq \SRW(\gamma)$. It suffices to show:
\begin{equation}\label{o3}
\sum_{\gamma: B_1^c\rightarrow K,\gamma\subseteq B}\SRW(\gamma)\Num(\gamma)\ll \|x\|^{-2}.
\end{equation}
By Cauchy-Schwarz inequality, one can get:
$$
\sum_{\gamma: B_1^c\rightarrow K,\gamma\subseteq B}\SRW(\gamma)\Num(\gamma)\leq
(\sum_{\gamma: B_1^c\rightarrow K,\gamma\subseteq B}\SRW(\gamma))^{1/2}
(\sum_{\gamma: B_1^c\rightarrow K,\gamma\subseteq B}\SRW(\gamma)(\Num(\gamma))^2)^{1/2}.
$$
By the Overshoot Lemma, the first term in the right hand side decays faster than any polynomial of $n$. On the other hand, due to \eqref{yy_1}, the second term in the right hand side is less than $\log n$ by a constant multiplier. Combining both gives \eqref{o3} and finishes the proof of \eqref{eq-1}.
\end{proof}

Now we are ready to prove Proposition \ref{thm-es}. Fix any small $\epsilon>0$. Let $n=\|x\|/ (\log \|x\|)^{1/4}$,
$$
\Gamma=\{\gamma:\Ball(2n)\setminus \Ball(n)\rightarrow K,\gamma\subseteq \Ball(n)\setminus K\},
$$
$$
\Gamma_1=\{\gamma\in \Gamma:|\Num(\gamma)-2a_4\sigma^2|K|\log n|>\epsilon \log n\}\quad \text{ and }\quad \Gamma_2=\Gamma\setminus \Gamma_1.
$$
By \eqref{yy_2}, we have: (when $\|x\|$ and hence $n$ are large)
\begin{equation}\label{yy_21}
\sum_{\gamma\in\Gamma_1}\SRW(\gamma)\preceq (\log n)^{-4}.
\end{equation}
Hence, we have (when $n$ is large):
\begin{multline*}
\sum_{\gamma\in\Gamma_1}\BRW(\gamma)\Num(\gamma)\leq\sum_{\gamma\in\Gamma_1}\SRW(\gamma)\Num(\gamma)
\leq(\sum_{\gamma\in\Gamma_1}\SRW(\gamma)\cdot
\sum_{\gamma\in\Gamma_1}\SRW(\gamma)(\Num(\gamma))^2)^{1/2}\\
\stackrel{\eqref{yy_21},\eqref{yy_1}}{\preceq}((\log n)^{-4}(\log n)^2)^{1/2}
=(\log n)^{-1}\ll |K|.
\end{multline*}
Combing this with \eqref{eq-1} gives:
$$
\sum_{\gamma\in\Gamma_2}\BRW(\gamma)\Num(\gamma)\sim|K|.
$$
Hence, we have (when $n$ is large):
$$
(1-\epsilon)|K|/\left((2a_4\sigma^2|K|+\epsilon)\log n\right)
\leq\sum_{\gamma\in\Gamma_2}\BRW(\gamma)
\leq (1+\epsilon)|K|/\left((2a_4\sigma^2|K|-\epsilon)\log n\right).
$$
On the other hand, $\sum_{\gamma\in\Gamma_1}\BRW(\gamma)\ll (\log n)^{-1}$.
Let $\epsilon\rightarrow 0^+$, one can get Proposition \ref{thm-es}.
\end{proof}

\begin{proof}[Proof of Lemma \ref{conv-3}]
Without loss of generality, assume $m=|a-b|$. Let $B_a=\{z:|z-a|\leq 3m/4\}$, $B_b=\{z:|z-b|\leq 3m/4\}$ and $B_c=\{z:|z-c|\leq M/4\}$. Write $t=(a+b)/2$ and $B=\{z:|z-t|\leq 2M\}$. Then we can estimate separately:
\begin{multline*}
\sum_{z\in B_a}|z-a|^{-2}|z-b|^{-2}|z-c|^{-2}\asymp \sum_{z\in B_a}\frac{1}{|z-a|^2m^2M^2}\\
\preceq \frac{1}{m^2M^2}\sum_{z\in B_a}\frac{1}{|z-a|^2}\asymp\frac{m^2}{m^2M^2}\leq \frac{1}{M^2};
\end{multline*}
\begin{equation*}
\sum_{z\in B_b}|z-a|^{-2}|z-b|^{-2}|z-c|^{-2}\preceq \frac{1}{M^2} \text{ (similarly)};
\end{equation*}
\begin{multline*}
\sum_{z\in B_c}|z-a|^{-2}|z-b|^{-2}|z-c|^{-2}\asymp \sum_{z\in B_a}\frac{1}{|z-c|^2M^2M^2}\\
\preceq \frac{1}{M^4}\sum_{z\in B_c}\frac{1}{|z-c|^2}\asymp\frac{M^2}{M^4}\leq \frac{1}{M^2};
\end{multline*}
\begin{multline*}
\sum_{z\in B\setminus(B_a\cup B_b\cup B_c)}|z-a|^{-2}|z-b|^{-2}|z-c|^{-2}\asymp
\sum_{ z\in B\setminus(B_a\cup B_b\cup B_c)}\frac{1}{|z-t|^2|z-t|^2M^2}\\
\preceq \frac{1}{M^2}\sum_{z:m/4\leq|z-t|\leq2M }\frac{1}{|z-t|^4}\asymp
\frac{1}{M^2}\sum_{:m/4\leq n\leq2M}\frac{n^3}{n^4}\preceq\frac{1\vee\log(M/m)}{M^2};
\end{multline*}
\begin{equation*}
\sum_{z\in B^c}|z-a|^{-2}|z-b|^{-2}|z-c|^{-2}\asymp \sum_{z\in B^c}\frac{1}{|z-t|^6}\preceq \sum_{n\geq 2M}\frac{n^3}{n^6}\preceq \frac{1}{M^2}.
\end{equation*}
This completes the proof.
\end{proof}

Now we are ready to prove Theorem \ref{MT4}.
\begin{proof}[Proof of Theorem \ref{MT4}]
Let $n=\|x\|/(\log \|x\|)^{1/4}$, $B=\Ball(n)$, $B_1=\Ball(2n)\setminus B$ and $B_2=\Ball(2n)^c$. As before, by \eqref{p2} and the First-Visit Lemma, we have:
\begin{multline*}
P(\Snake_x \text { visits }K)=\sum_{\gamma: x\rightarrow K}\BRW(\gamma)=\sum_{b\in B^c} G_K(x,b)
\sum_{a\in K}\Hm^{B}_\KRW(b,a)\\
=\sum_{b\in B_1} G_K(x,b)\sum_{a\in K}\Hm^{B}_\KRW(b,a)+\sum_{b\in B_2} G_K(x,b)\sum_{a\in K}\Hm^{B}_\KRW(b,a).
\end{multline*}
We argue that the first term has the desired asymptotics and the second is negligible:
\begin{align*}
\sum_{b\in B_1} &G_K(x,b)\sum_{a\in K}\Hm^{B}_\KRW(b,a)
\stackrel{\eqref{G_local}}{\sim} a_4\|x\|^{-2}\sum_{b\in B_1}\sum_{a\in K}\Hm^{B}_\KRW(b,a)\\
&\stackrel{\eqref{thm_es}}{\sim} a_4\|x\|^{-2} \frac{4\pi^2\sqrt{\det Q}}{\sigma^2\log n}
\sim\frac{1}{2\sigma^2\|x\|^2\log \|x\|};
\end{align*}
\begin{equation*}
\sum_{b\in B_2} G_K(x,b)\sum_{a\in K}\Hm^{B}_\KRW(b,a)\preceq \sum_{a\in K}\sum_{b\in B_2}\Hm^{B}_\KRW(b,a)
\stackrel{\eqref{os}}{\preceq} |K|n^2/n^{5}\ll 1/\|x\|^{2}\log\|x\|.
\end{equation*}

\end{proof}

\section{Convergence of the first visiting point.}

We aim to show Theorem \ref{MT6}. For simplicity, we assume in this section that $\theta$ has finite range. Therefore, for any subset $B\ssubset\Z^4$, we can define its outer boundary and inner boundary by:
$$
\outBd B\doteq\{y\notin B:\exists x\in B, \text{ such that }\theta(x-y)\vee\theta(y-x)>0\};
$$
$$
\intBd B\doteq\{y\in B:\exists x\notin B, \text{ such that }\theta(x-y)\vee\theta(y-x)>0\}.
$$

The first step is to construct the following asymptotical behavior of the Green function:

\begin{lemma}\label{G-4d}
\begin{equation}\label{G_4d}
\lim_{x,y\rightarrow \infty:\|x\|\geq\|y\|}\frac{G_K(x,y)}{(\log \|y\|/\log \|x\|)g(x,y)}=1;
\end{equation}
\end{lemma}
\begin{remark}
It is a bit unsatisfactory that we need to require $\|x\|\geq \|y\|$ in the limit. When $\theta$ is symmetric, this requirement can be removed  since $G_K(x,y)/(1-\KRW(x))=G_K(y,x)/(1-\KRW(y))$.
\end{remark}
\begin{proof}
By Lemma \ref{con-G-local}, we can assume $\|x\|\geq \|y\|(\log \|y\|)^{1/4}$. Let $n=\|y\|(\log \|y\|)^{1/8}$ and $B=\Ball(n)$.
As before, we have:
$$
\pS_K(x)=G_K(x,K)=\sum_{z\in \intBd B}\Hm_\KRW^{B^c}(x,z)G_K(z,K)=\sum_{z\in \intBd B}\Hm_\KRW^{B^c}(x,z)\pS_K(z).
$$
By Theorem \ref{MT4}, we get:
\begin{equation}
\sum_{z\in \intBd B}\Hm_\KRW^{B^c}(x,z)
\sim\frac{n^2\log n}{\|x\|^2\log \|x\|}.
\end{equation}
By Lemma \ref{con-G-local}, we have $G_K(z,y)\sim g(z,y)\sim a_4n^{-2}$ for any $z\in \intBd B$. Therefore,
\begin{multline*}
G_K(x,y)=\sum_{z\in \intBd B}\Hm_\KRW^{B^c}(x,z)G_K(z,y)
\sim \sum_{z\in \intBd B}\Hm_\KRW^{B^c}(x,z)a_4n^{-2}\\
\sim a_4\log n/(\|x\|^2\log\|x\|)\sim a_4\log \|y\|/(\|x\|^2\log\|x\|).
\end{multline*}
This finishes the proof.
\end{proof}

Now we give the following asymptotics of the escape probability by a reversed snake.
\begin{lemma}
For any $x\in \Z^4$, we have:
\begin{equation}\label{def-EE}
\EE_K(x)\doteq \lim_{n\rightarrow \infty }\log n\cdot
\sum_{z\in\outBd \Ball(n)}\Hm_\KRW^{\Ball(n)}(z,x) \text{ exists}.
\end{equation}
\end{lemma}
\begin{remark}
Note that $\Hm_\KRW^{\Ball(n)}(z,x)=\Hm_\KRW^{\Ball(n)\setminus K}(z,x)$ and $\sum_{z\in\outBd \Ball(n)}\Hm_\KRW^{\Ball(n)}(z,x)$ is roughly the probability that a reversed snake starting from $x$ does not return to $K$, except for the bush grafted to the root, until the backbone reaches outside of $\Ball(n)$. For the random walk in critical dimension ($d=2$), we also have (e.g. see Section 2.3 in \cite{L91}):
$$
E_K(x)\doteq \lim_{n\rightarrow \infty }\log n\cdot
\sum_{z\in\outBd \Ball(n)}\Hm^{\Ball(n)\setminus K}(z,x) \text{ exists},\quad \text{for any }x\in\Z^2, K\ssubset \Z^2;
$$
and
$$
\lim_{x\rightarrow \infty}P(S_x(\tau_K)=a|S_x \text{ visits } K)=
\frac{1}{\pi^2\sqrt{\det Q}}E_K(a).
$$
\end{remark}
\begin{proof}
We first need to show:
\begin{equation}\label{cr1}
\lim_{n\rightarrow\infty, y\rightarrow \infty:\;\|y\|\leq n}\frac{\log n}{\log \|y\|}
\sum_{z\in\outBd \Ball(n)}\Hm_\KRW^{\Ball(n)}(z,y)=1.
\end{equation}
Choose some $x\in\Z^4$ such that $\|x\|\geq n\log n$. By the First-Visit Lemma, we have:
\begin{equation}\label{cr2}
G_K(x,y)=\sum_{z\in\outBd \Ball(n)}G_K(x,z)\Hm_\KRW^{\Ball(n)}(z,y).
\end{equation}
Due to the last Lemma, $G_K(x,y)\sim a_4\|x\|^{-2}\cdot \log\|y\|/\log\|x\|$,
$G_K(x,z)\sim a_4\|x\|^{-2}\cdot \log n/\log\|x\|$. Together with \eqref{cr2}, one can get \eqref{cr1}.

Now we are ready to show \eqref{def-EE}. Without loss of generality, assume $\|x\|>\Rad(K)$. Write
$$
a(n)=\log n\cdot
\sum_{z\in\outBd \Ball(n)}\Hm_\KRW^{\Ball(n)}(z,x).
$$
Note that, for any (large) $m> n$,
$$
\sum_{w\in\outBd \Ball(m)}\Hm_\KRW^{\Ball(m)}(w,x)=
\sum_{z\in\outBd \Ball(n)}\Hm_\KRW^{\Ball(n)}(z,x)\sum_{w\in\outBd \Ball(m)}\Hm_\KRW^{\Ball(m)}(w,z).
$$
By \eqref{cr1}, we have $\sum_{w\in\outBd \Ball(m)}\Hm_\KRW^{\Ball(m)}(w,z)\sim \log n/\log m$. This implies
$a(n)/a(m)\sim1$ and hence the convergence of $a(n)$.
\end{proof}

\begin{proof}[Proof of Theorem \ref{MT5}]
Let $n=\|x\|/\log \|x\|$ and $B=\Ball(n)$. Then,
\begin{align*}
&P(\Snake_x(\tau_K)=a|\Snake_x \text{ visits } K)=
\frac{\sum_{\gamma:x\rightarrow a}\BRW(\gamma)}{\pS_K(x)}
\stackrel{\eqref{MT_4}}{\sim}
 \frac{\sum_{z\in\outBd B}G_K(x,z)\Hm_\KRW^B(z,a)}{1/2\sigma^2\|x\|^2\log \|x\|}\\
&\sim \frac{a_4\|x\|^{-2}\sum_{z\in\outBd B}\Hm_\KRW^B(z,a)}{1/2\sigma^2\|x\|^2\log \|x\|}
\sim \frac{a_4\|x\|^{-2}\EE_K(a)\log^{-1} n}{1/2\sigma^2\|x\|^2\log \|x\|}
\sim 2\sigma^2a_4\EE_K(a)=\frac{\sigma^2\EE_K(a)}{4\pi^2\sqrt{\det Q}}.\\
\end{align*}
\end{proof}

\section{The range of branching random walk conditioned on the total size.}
The main goal of this section is to construct the asymptotics of the range of the branching random walk conditioned on the total size, i.e. Theorem \ref{MT5}. Our proof of this theorem (and some other results in this subsection) is based on some ideas from \cite{LL141}. Especially, we need to use the invariant shift on the invariant snake, $\Snake^I$.

For the invariant $\Snake^I$, recall that its backbone is just a random walk. We write $\tau_n$ for the hitting time (vertex) of $(\Ball(n))^c$ by the backbone. Thanks to Proposition \ref{thm-es}, we can obtain the following:
\begin{prop}\label{esc-4d}
$$
P(\Snake^I_0(v)\neq0,\; \forall v \leq \tau_n \text { not on the spine} )\sim \frac{4\pi^2\sqrt{\det Q}}{\sigma^2}\frac{1}{\log n};
$$
$$
P(\Snake^I_0(v_i)\neq0,\; i=1,2,\dots,n )\sim \frac{16\pi^2\sqrt{\det Q}}{\sigma^2}\frac{1}{\log n};
$$
where $v_1< v_2<v_3<\dots$ are all vertices of $\Snake^I_0$ that are not on the spine.
\end{prop}
\begin{proof}
By Proposition \ref{thm-es} (set $K=\{0\}$) and the Overshoot Lemma, we have:
$$
\sum_{\gamma:(\Ball(n))^c\rightarrow 0,\gamma\subseteq \Ball(n)}\BRW(\gamma)\sim
\frac{4\pi^2\sqrt{\det Q}}{\sigma^2}\frac{1}{\log n}.
$$
Hence, the first assertion can be obtained if we can show
$$
P(\Snake^I_0(v)\neq0,\; \forall v \leq \tau_n  \text { not on the spine} )\sim
\sum_{\gamma:(\Ball(n))^c\rightarrow 0,\gamma\subseteq \Ball(n)}\BRW(\gamma).
$$

Let $p_0=P(\Snake_0\text{ does not visit }0\text{ except at the root})$ and the new killing function $k'(x)$ be the probability that $\Snake'_x$ returns to $0$ (except possibly for the starting point). Note that $k'(x)=\KRW_K(x)$ when $x\neq 0$. We write $\BRW_{k'}(\gamma)$ for the probability weight of $\gamma$ with this killing function.
Then, we have
\begin{align*}
P&(\Snake^I_0(v)\neq0,\; \forall v \leq \tau_n \text { not on the spine} )\sim
p_0\sum_{\gamma:(\Ball(n))^c\rightarrow 0,\;\gamma\subseteq\Ball(n)}\BRW_{k'}(\gamma)\\
&=p_0(\sum_{\gamma:(\Ball(n))^c\rightarrow 0,\;\gamma\subseteq\Ball(n)\setminus\{0\}}\BRW_{k'}(\gamma))
(\sum_{\gamma:0\rightarrow 0,\;\gamma\subseteq\Ball(n)}\BRW_{k'}(\gamma)).\\
\end{align*}
Note that $\lim_{n\rightarrow \infty}\sum_{\gamma:0\rightarrow 0,\;\gamma\subseteq\Ball(n)}\BRW_{k'}(\gamma)=
\sum_{\gamma:0\rightarrow 0}\BRW_{k'}(\gamma)$ and
$$
\sum_{\gamma:(\Ball(n))^c\rightarrow 0,\;\gamma\subseteq\Ball(n)}\BRW(\gamma)=
\sum_{\gamma:(\Ball(n))^c\rightarrow 0,\;\gamma\subseteq\Ball(n)\setminus\{0\}}\BRW_{k'}(\gamma).
$$
Hence, for the first assertion, it is sufficient to show:
\begin{equation}\label{a3}
p_0\sum_{\gamma:0\rightarrow 0}\BRW_{k'}(\gamma)=1.
\end{equation}
Note that \eqref{p2} is obtained from the viewpoint of 'the first visiting point'. In fact we also have an analogous formula from the viewpoint of 'the last visiting point' (see Section 5 in \cite{Z161}) which implies that:
$$
P(\Snake_0 \text{ visits } K)=p_0\sum_{\gamma:0\rightarrow K}\BRW_{k'}(\gamma).
$$
This is just \eqref{a3} (note that $K=\{0\}$) and finishes the proof of the first assertion. The second assertion is an easy consequence of the first one, noting that, for any $\epsilon\in(0,1/4)$ fixed, $P(v_n\leq \tau_{\lfloor n^{1/4-\epsilon}\rfloor})$ and ,$P(v_n \geq \tau_{\lfloor n^{1/4+\epsilon}\rfloor})$ are $o((\log n)^{-1})$.
\end{proof}

Now we can construct the following result about the range of $S^I$:

\begin{theorem}\label{ran-1}
Set $R^I_n:=\#\{\Snake_0^I(o),\Snake_0^I(v_1),\dots,\Snake_0^I(v_n)\}$ for every integer $n\geq 0$. We have:
$$
\frac{\log n}{n}R^I_n\stackrel{L^2}{\longrightarrow} \frac{16\pi^2\sqrt{\det Q}}{\sigma^2}
\quad\text{as } n\rightarrow\infty,
$$
where $v_1,v_2,\dots$ are the same as in Proposition \ref{esc-4d}. Hence, we have:
$$
\frac{\log n}{n}R^I_n\stackrel{P}{\longrightarrow} \frac{16\pi^2\sqrt{\det Q}}{\sigma^2}
\quad\text{as } n\rightarrow\infty.
$$
\end{theorem}
\begin{remark}
Since the typical number of vertices in the spine that come before $v_n$ is of order $\sqrt{n}$, which is much less than $n/\log n$, one can get,
$$
\frac{\log n}{n} \#\{\Snake_0^I(\bar{v}_0),\Snake_0^I(\bar{v}_1),\dots,\Snake_0^I(\bar{v}_n)\} \stackrel{P}{\longrightarrow} \frac{16\pi^2\sqrt{\det Q}}{\sigma^2}
\quad\text{as } n\rightarrow\infty,
$$
where $\bar{v}_0,\bar{v}_1,\dots$ are all vertices due to the default order in the corresponding plane tree $T$ in $\Snake^I_0$.
\end{remark}

\begin{proof}[Proof of Theorem \ref{ran-1}]
As mentioned before, we need to use the invariant shift $\varsigma$ on spacial trees, which appeared in \cite{LL141}. For any spacial tree $(T, \Snake_T)$, set $\varsigma(T, \Snake_T)=(T',\Snake'_{T'})$. Roughly speaking, one can get $T'$ by 'rerooting' $T$ at the first vertex that is not in the spine and then removing the vertices that are strict before the parent of the new root. For $\Snake'_{T'}$, just set:
$$
\Snake'_{T'}(v)=\Snake_T(v)-\Snake_T(o'), \quad\text{ for any }v\in \varsigma(T),
$$
where $o'$ is the new root.  The key result is that $\varsigma$ is invariant under the law of the invariant snake from the origin. For more details about this shift transformation, see Section 2 in \cite{LL141}.

Now we start our proof. For simplicity, write $\hat{v}_0=0(\in\Z^4)$ and  $\hat{v}_i=\Snake_0^I(v_i)$. First observe that:
$$
E(R^I_n)=E(\sum_{i=0}^n\textbf{1}_{\{\hat{v}_j\neq\hat{v}_i,\forall j\in[i+1,n]\}})
=\sum_{i=0}^nP(\hat{v}_j\neq\hat{v}_i,\forall j\in[i+1,n]).
$$
From the invariant shift mentioned in the beginning, we have
$$
P(\hat{v}_j\neq\hat{v}_i,\forall j\in[i+1,n])=P(\hat{v}_j\neq\hat{v}_0,\forall j\in[1,n-i]).
$$
Therefore by Proposition \ref{esc-4d}, we get
\begin{equation}\label{ran1}
E(R^I_n)=\sum_{i=0}^nP(\hat{v}_j\neq\hat{v}_0,\forall j\in[1,n-i])
\sim \frac{16\pi^2\sqrt{\det Q}}{\sigma^2}\frac{n}{\log n}.
\end{equation}
Now we turn to the second moment. Similarly, we have
\begin{align*}
E((R^I_n)^2)&=E(\sum_{i=0}^n\sum_{j=0}^n\textbf{1}_{\{\hat{v}_k\neq\hat{v}_i,\forall k\in[i+1,n];
\hat{v}_l\neq\hat{v}_j,\forall l\in[j+1,n]\}})\\
&=2\sum_{0\leq i<j\leq n}P(\hat{v}_k\neq\hat{v}_i,\forall k\in[i+1,n];
\hat{v}_l\neq\hat{v}_j,\forall l\in[j+1,n])+E(R_n)\\
&=2\sum_{0\leq i<j\leq n}P(\hat{v}_k\neq0,\forall k\in[1,n-i];
\hat{v}_l\neq\hat{v}_{j-i},\forall l\in[j-i+1,n-i])\\
&\quad+E(R_n),\\
\end{align*}
where the last equality again follows from the invariant shift. For any fixed $\alpha\in(0,1/4)$ define
$$
\sigma_n:=\sup\{k\geq0: v_k\leq u_{\lfloor n^{\frac{1}{2}-\alpha}\rfloor}\},
$$
where $u_0\leq u_1\leq\dots$ are the all vertices on the spine.
By standard arguments, one can show
$$
P(\sigma_n\notin[n^{1-3\alpha}, n^{1-\alpha}])=o(\log^{-2} n).
$$
Therefore we have
\begin{multline*}
\limsup_{n\rightarrow \infty}(\frac{\log n}{n})^2E((R^I_n)^2)=\limsup_{n\rightarrow \infty}
2(\frac{\log n}{n})^2\sum_{0\leq i<j\leq n}P(\hat{v}_k\neq0,\forall k\in[1,n-i];\\
\hat{v}_l\neq\hat{v}_{j-i},\forall l\in[j-i+1,n-i];\sigma_n\in[n^{1-3\alpha}, n^{1-\alpha}]).
\end{multline*}
Obviously, in order to study the limsup in the right-hand side, we can restrict the sum to indices $i$ and $j$ such that $j-i>n^{1-\alpha}$. However, when $i$ and $j$ are fixed and satisfied with $j-i>n^{1-\alpha}$,
\begin{align*}
P&(\hat{v}_k\neq0,\forall k\in[1,n-i];
\hat{v}_l\neq\hat{v}_{j-i},\forall l\in[j-i+1,n-i];\sigma_n\in[n^{1-3\alpha}, n^{1-\alpha}])\\
&\leq P(\hat{v}_k\neq0,\forall k\in[1,\sigma_n];
\hat{v}_l\neq\hat{v}_{j-i},\forall l\in[j-i+1,n-i];\sigma_n\in[n^{1-3\alpha}, n^{1-\alpha}])\\
&=P(\hat{v}_k\neq0,\forall k\in[1,\sigma_n];\sigma_n\in[n^{1-3\alpha}, n^{1-\alpha}])
P(\hat{v}_l\neq\hat{v}_{j-i},\forall l\in[j-i+1,n-i])\\
&=P(\hat{v}_k\neq0,\forall k\in[1,\sigma_n];\sigma_n\in[n^{1-3\alpha}, n^{1-\alpha}])
P(\hat{v}_l\neq0,\forall l\in[1,n-j]).
\end{align*}
Note that for the second last line, we use the fact that after conditioning on $\sigma_n=m(<n^{1-\alpha})$, the event on the second probability is independent to the event on the first one, and for the last line, we use the invariant shift. Now,
$$
P(\hat{v}_k\neq0,\forall k\in[1,\sigma_n];\sigma_n\in[n^{1-3\alpha}, n^{1-\alpha}])
\leq P(\hat{v}_k\neq0,\forall k\in[1,n^{1-3\alpha}]),
$$
and then we have
\begin{align*}
\limsup_{n\rightarrow \infty}&(\frac{\log n}{n})^2E((R^I_n)^2)
\leq \limsup_{n\rightarrow \infty}2(\frac{\log n}{n})^2\cdot\\
&\sum_{0\leq i<j\leq n, j-i>n^{1-\alpha}}
P(\hat{v}_k\neq0,\forall k\in[1,n^{1-3\alpha}])P(\hat{v}_l\neq0,\forall l\in[1,n-j])\\
&=\frac{1}{1-3\alpha}(\frac{16\pi^2\sqrt{\det Q}}{\sigma^2})^2.
\end{align*}
Let $\alpha\rightarrow 0^+$, we get
$$
\limsup_{n\rightarrow\infty}(\frac{\log n}{n})^2E((R^I_n)^2)\leq(\frac{16\pi^2\sqrt{\det Q}}{\sigma^2})^2.
$$
Combining this with \eqref{ran1}, we finish the proof of Theorem \ref{ran-1}.
\end{proof}

Noting that $\Snake_0^-$ is different to $\Snake_0^I$ only at the subtree grafted to the root, one can also obtain the range of the infinite snake $\Snake^-$:
\begin{corollary}\label{ran-2}
Set $R^-_n:=\#\{\Snake_0^-(v_0),\Snake_0^-(v_1),\dots,\Snake_0^-(v_n)\}$. Then,
$$
\frac{\log n}{n}R^-_n\stackrel{P}{\longrightarrow} \frac{16\pi^2\sqrt{\det Q}}{\sigma^2}
\quad\text{as } n\rightarrow\infty,
$$
where $v_0,v_1,\dots$ are all vertices of corresponding plane tree due to the default order in the reversed snake.
\end{corollary}

Now we are ready to prove our main result about the range of branching random walk conditioned on the total size. This result will follow from Corollary \ref{ran-2} by an absolute continuity argument, which is similar to the one in the proof of Theorem 7 in \cite{LL141}. The idea is as follows. We write $\pr$ for the law of the $\mu$-GW tree. For every $a\in(0,1)$, the law under $\pr^n:=\pr(\cdot|\#T=n)$ of the
subtree obtained by keeping only the first $\lfloor an\rfloor$ vertices of $T$ is absolutely continuous with
respect to the law under $\pr^\infty(\cdot):=\pr(\cdot|\#T=\infty)$ of the same subtree, with a density that is bounded independently of $n$. Then a similar property holds for spatial trees, and hence we can use the convergence in Corollary \ref{ran-2}, for a tree distributed according to $\pr^\infty$, to get a similar convergence for a tree distributed according to $\pr^n$.

\begin{proof}[Proof of Theorem \ref{MT5}]
Let $\mathcal{G}$ be the smallest subgroup of $\Z$ that contains the support of $\mu$. In fact, the cardinality of the vertex set of a $\mu$-GW tree belongs to $1+\mathcal{G}$. For simplicity, we assume in the proof that $\mathcal{G}=\Z$. Minor modifications are needed for the general case. On the other hand, for any sufficiently large integer $n\in 1+\mathcal{G}$, we can define the conditional probability $\Snake^n$ to be $\Snake_0$ conditioned on the total number of vertices being $n$ (this event is with strictly positive probability).

For a finite plane tree $T$, write $v_0(T), v_1(T),\dots, v_{\#T-1}(T)$ for the vertices of $T$ by the default order. The Lukasiewisz path of $T$ is then the finite sequence $(X_l(T),0\leq l\leq \#T)$, which can be defined inductively by
$$
X_0(T)=0, X_{l+1}-X_l=k_{v_l(T)}(T)-1, \quad \text{for every }0\leq l<\#T,
$$
where $k_u(T)$(for $u\in T$) is the number of children of $u$. The tree $T$ is determined by its Lukasiewisz path. A key result says that under $\pr$, the Lukasiewisz path is distributed as a random walk on $\Z$ with jump distribution $\nu$ determined by $\nu(j)=\mu(j+1)$ for any $j\geq-1$, which starts from $0$ and is stopped at the hitting time of $-1$ (in particular, the law or $\#T$ coincides with the law of that hitting time). For notational convenience, we let $(Y_k)_{k\geq0}$ be a random walk on $\Z$ with jump distribution $\nu$, which starts form $i$ under $P_{(i)}$, and set
$$
\tau:=inf\{k\geq0:Y_k\leq-1\}.
$$

We can also do this for infinite trees. When $T$ ia an infinite tree with only one infinite ray, now the Depth-First search sequence $o=v_0<v_1<v_2<\dots<v_n<\dots$ only examines part of the vertex set of $T$. We could also define the Lukasiewisz path of $T$ to be the infinite sequence $(X_i(T),i\in\N)$:
$$
X_0(T)=0, X_{l+1}-X_l=k_{v_l(T)}(T)-1, \quad \text{for every }l\in\N.
$$
Now, only the 'left half' of $T$ (precisely, the subtree generated by $v_0,v_1,...$), no the whole tree $T$, is determined by its Lukasiewisz path. It is not difficult to verify that when $T$ is a $\mu$-GW tree conditioned on survival, its Lukasiewisz path is distributed as the random walk on the last paragraph conditioned on $\tau=\infty$, i.e, a Markov chain on $\N$ with transition probability $p(i,j)=\frac{j+1}{i+1}\nu(j-i)$. Recall that the infinite $\mu$-GW tree is just the 'left half' of the $\mu$-GW tree conditioned on survival.

Next, take $n$ large enough such that $\pr(\#T=n)>0$. Fix $a\in(0,1)$, and consider a tree (finite or infinite) $T$ with $\# T>n$. Then, the collection of vertices $v_0(T),\dots, v_{\lfloor an\rfloor}(T)$ forms a subtree of $T$ (because in the Depth-First search order the parent of a vertex comes before this vertex), and we denote this tree by $\subt_{\lfloor an\rfloor}(T)$. It is elementary to see that $\subt_{\lfloor an\rfloor}(T)$ is determined by the sequence $(X_l(T),0\leq l\leq \lfloor an\rfloor)$. Let $f$ be a bounded function on $\Z^{\lfloor an\rfloor}$. One can verify that
\begin{equation}\label{rr1}
\pr^n(f((X_k)_{0\leq k\leq \lfloor an\rfloor)})=\frac{1}{P_{(0)}(\tau=n+1)}\pr^\infty(f((X_k)_{0\leq k\leq \lfloor an\rfloor})
\frac{\psi_n(X_{\lfloor an\rfloor})}{X_{\lfloor an\rfloor}+1}),
\end{equation}
where for every $j\in N$, $\psi_n(j)=P_{(j)}(\tau=n+1-\lfloor an\rfloor)$.

We now let $n\rightarrow \infty$. Using Kemperman's formula and a standard local limit theorem, one can get,
\begin{equation}\label{rr2}
\lim_{n\rightarrow \infty}\left(\sup_{j\in A_n}|\frac{\psi_n(j)}{P_{(0)}(\tau=n+1)(j+1)}-\Gamma_a(\frac{j}{\sigma\sqrt{n}})|\right)=0,
\end{equation}
where $\Gamma_a(x)=\exp(-\frac{x^2}{2(1-a)})/(1-a)^{\frac{3}{2}}$ and $A_n:=\{i\in \N:P_{(i)}(\tau=n+1-\lfloor an\rfloor)>0\}$.
By combining \eqref{rr1} and \eqref{rr2}, we get that, for any uniformly bounded sequence of functions $(f_n)_{n\geq1}$ on $\Z^{\lfloor an\rfloor+1}$, we have
$$
\lim_{n\rightarrow\infty}|\pr^n(f_n((X_k)_{0\leq k\leq \lfloor an\rfloor}))-
\pr^\infty(f_n((X_k)_{0\leq k\leq \lfloor an\rfloor})\Gamma_a(\frac{X_{\lfloor an\rfloor}}{\sigma\sqrt{n}}))|=0.
$$
Clearly, the above still holds after we add the spatial random mechanism.  Therefore, when $\epsilon>0$ is fixed, we have
\begin{multline*}
\lim_{n\rightarrow\infty}|\pr_{\theta}^n(\textbf{1}_{\{|R_{\lfloor an\rfloor}-tan/\log n|
>\epsilon n/\log n\}})-\\
\pr_{\theta}^\infty(\textbf{1}_{\{|R_{\lfloor an\rfloor}-tan/\log n|
>\epsilon n/\log n\}}\Gamma_a(\frac{X_{\lfloor an\rfloor}}{\sigma\sqrt{n}}))|=0,
\end{multline*}
where $t=\frac{16\pi^2\sqrt{\det Q}}{\sigma^2}$, $\pr_{\theta}^n$, $\pr_{\theta}^\infty$ are the laws of the corresponding tree-indexed random walks, and $R_{\lfloor an\rfloor}$ is the range of the subtree $\subt_{\lfloor an\rfloor}(T)$. Note that the function $\Gamma_a$ is bounded and under $\pr_{\theta}^\infty$, $R_{\lfloor an\rfloor}$ is just the range of $\Snake^\infty_0$ for the first $\lfloor an\rfloor$ vertices. Hence, by Corollary \ref{ran-2} (note that $\Snake^-_0=\Snake^{\infty}_0$ since we assume that $\theta$ is symmetry), we obtain that
$$
\lim_{n\rightarrow\infty}\pr_{\theta}^n(\textbf{1}_{\{|R_{\lfloor an\rfloor}-tan/\log n|
>\epsilon n/\log n\}})=0.
$$
Note that $R_n\geq R_{\lfloor an\rfloor}$ (under $\pr_{\theta}^n$) and $a$ can be chosen arbitrarily close to $1$, this finishes the proof of the lower bound.

We also need to show the upper bound. Note that $\subt_{\lfloor an\rfloor}(T)$ is the subtree lying on the 'left' side, generated by the first $\lfloor an\rfloor$ vertices of $T$. Similarly, one can consider the subtree lying on the 'right' side. Strictly speaking, to get the subtree lying on the right side, denoted by $\subt^-_{\lfloor an\rfloor}(T)$, we first reverse the order of children for each vertex in $T$, and then $\subt_{\lfloor an\rfloor}$ of the same tree $T$ with the new order is just $\subt^-_{\lfloor an\rfloor}(T)$. Write $R^-_{\lfloor an\rfloor}$ for the range of $\subt^-_{\lfloor an\rfloor}(T)$ corresponding to $\Snake^n$. By symmetry, we also have
$$
\lim_{n\rightarrow\infty}\pr_{\theta}^n(\textbf{1}_{\{|R^-_{\lfloor an\rfloor}-tan/\log n|
>\epsilon n/\log n\}})=0.
$$

Now fix some $a\in(0,1)$. Note that $\subt_{\lfloor an\rfloor}(T)$ and $\subt^-_{\lfloor (1-a)n\rfloor}(T)$ cover the whole tree $T$ except for a number of vertices. This number is not more that $|\subt_{\lfloor an\rfloor}(T)\cap\subt^-_{\lfloor (1-a)n\rfloor}(T)|+2$. Note that on each generation, there is at most one vertex that is in both $\subt_{\lfloor an\rfloor}(T)$ and $\subt^-_{\lfloor (1-a)n\rfloor}(T)$. Hence
$|\subt_{\lfloor an\rfloor}(T)\cap\subt^-_{\lfloor (1-a)n\rfloor}(T)|$ is not more than the number of generations, which is typically of order $\sqrt{n}$ (under $\pr^n$). Hence, $R_n-(R_{\lfloor an\rfloor}(\Snake^n)+R^-_{\lfloor (1-a)n\rfloor}(\Snake^n))$ is less than $n^{0.6}$ with high probability (tending to 1). This finishes the proof of the upper bound.
\end{proof}

\section*{Acknowledgements}

We thank Professor Omer Angel for valuable comments on earlier versions of this paper.

\end{document}